\documentclass [11pt,twoside,a4paper]{article}
\usepackage{amsfonts}
\usepackage{amsthm}
\usepackage{amsmath}
\usepackage{amscd}
\usepackage{amsbsy}            %
\usepackage{psfrag}            %
\usepackage{epsf}              %
\usepackage{graphicx}          %
\usepackage{makeidx}           %
\usepackage{color}             %
\usepackage{fancyhdr}

\newtheorem{thm}{Theorem}[section]
\newtheorem{exm}[thm]{Example}

\newtheorem{lem}[thm]{Lemma}

\newtheorem{cor}[thm]{Corollary}

\newcommand{\abs}[1]{\left\vert#1\right\vert}
\newcommand{\aut}{\textit{Aut}}
\newcommand{\upfit}[1]{\lceil#1\rceil}
\newcommand{\lowfit}[1]{\lfloor#1\rfloor}
\newcommand{\dbra}[1]{[\![#1]\!]}

\newcommand{\ffi}{ \Leftrightarrow\ }

\newcommand{\rank}{\texttt{rank}}

\newcommand{\R}{\mathit{R}}
\newcommand{\RPlus}{\textit{R}_{+}}

\newcommand{\set}[1]{\left\{#1\right\}}

\newcommand{\tr}{\texttt{Tr}}
\newcommand{\vecc}{\texttt{vec}}
\newcommand{\vecr}{\texttt{vecr}}

\newcommand{\al}{\alpha}

\newcommand{\be}{\beta}
\newcommand{\hal}{\hat{\al}}

\newcommand{\la}{\lambda}

\newcommand{\p}{\prime}
\newcommand{\eps}{\epsilon}
\newcommand{\beps}{\bf{\eps}}
\newcommand{\si}{\sigma}

\newcommand{\KK}{\mathbb{K}}

\newcommand{\A}{\mathcal{A}}
\newcommand{\tA}{\tilde{\mathcal{\A}}}
\newcommand{\B}{\mathcal{B}}
\newcommand{\C}{\mathcal{C}}

\newcommand{\F}{\mathcal{F}}
\newcommand{\G}{\mathcal{G}}
\newcommand{\GL}{\mathcal{G\!L}}
\newcommand{\caH}{\mathcal{H}}
\newcommand{\caI}{\mathcal{I}}

\newcommand{\K}{\mathcal{K}}

\newcommand{\ST}{\mathcal{S\!T}}
\newcommand{\T}{\mathcal{T}}

\newcommand{\X}{\mathcal{X}}
\newcommand{\Y}{\mathcal{Y}}

\newcommand{\bx}{\textbf{x}}

\newcommand{\by}{\textbf{y}}

\newcommand{\bfe}{\textbf{e}}

\newcommand{\mnrts}{$m$th order $n$-dimensional real tensors }

\newcommand{\mnsts}{$m$th order $n$-dimensional symmetric tensors }

\newcommand{\beq}{\begin{equation}}
\newcommand{\eeq}{\end{equation}}
\newcommand{\bey}{\begin{eqnarray}}
\newcommand{\eey}{\end{eqnarray}}
\newcommand{\beyy}{\begin{eqnarray*}}
\newcommand{\eeyy}{\end{eqnarray*}}

\title{Commutation matrices and Commutation tensors$^{\ast}$}
\author{Changqing Xu\thanks{School of Mathematics and Physics, Suzhou University of Science and Technology, Suzhou, China. Email: cqxurichard@usts.edu.cn.}
\footnote{This work was partially presented at the international conference on Matrix analysis and Applications(ICMAA2018), Nagano, Japan in June 22-24, 2018, and the 2018 workshop on Matrices and Operators (MAO2018), Shanghai in July 15-17, 2018.}}

  \makeatletter
      \def\@setcopyright{}
      \def\serieslogo@{}
      \makeatother

\begin{document}

\maketitle

\begin{abstract}
The commutation matrix was first introduced in statistics as a transposition matrix by Murnaghan in 1938. In this paper, we first investigate the commutation matrix 
which is employed to transform a matrix into its transpose.  We then extend the concept of the commutation matrix to commutation tensor and use the commutation 
tensor to achieve the unification of the two formulae of the linear preserver of the matrix rank, a classical result of Marcus in 1971.  
\end{abstract}

\noindent \textbf{keywords:} \  Commutation matrix; commutation tensor;  Linear preserver;  determinant;  transpose.

\section{Introduction}
\setcounter{equation}{0}

The commutation matrix was introduced by Murnaghan in 1938 in the name of permutation matrices. It is also referred in publications on statistics 
as the transposition matrix.  A commutation matrix is a kind of permutation matrix of order $pq$ expressed as a block matrix where each block is of the 
same size and has a unique 1 in it.  The commutation matrix can be used to describe the relationship of a Kronecker product $A\otimes B$ with $B\otimes A$
where $A,B$ are two arbitrary matrices of any sizes. In this paper, we extend the commutation matrix to a commutation tensor, which is a fourth order tensor,
by which we can express the transpose of a matrix as the linear transformation on it. This is further used to deduce some properties of the commutation tensors,
and consequently we achieve the unification of the linear preservers of the determinants of matrices. \\
\indent For our convenience, we denote $\dbra{m}$ for the set $\set{1,2,\ldots,m}$ for any positive integer, and $S_{m}$ the set of all the permutations on 
$\dbra{m}$. Also we denote by $[a/b]$ the quotient of an integer $a$ divided by a positive integer $b$, similarly we denote $\upfit{a/b}$ ($\lowfit{a/b}$) for 
the upper (resp. lower) quotient of an integer $a$ divided by a positive integer $b$. We write $n(p,q)=pq$ or simply $n=pq$ when no risk of confusion arises for 
any positive integers $p,q$.  Denote $e_{n,i}$ for the $i$th canonical vector of dimension $n$, i.e., the vector with 1 in the $i$th coordinate and 0 elsewhere, and 
denote $E^{(m,n)}_{ij}=e_{m,i}e_{n,j}^{\top}$, and $E_{ij}^{(n)}=E^{(m,n)}_{ij}$ when $m=n$.  We usually denote them by $E_{ij}$ when $m,n$ are known from the 
context.  A permutation matrix  $P=(p_{ij})\in \R^{n\times n}$ is called a \emph{commutation matrix} if it satisfies the following conditions:
\begin{description}
  \item[(a)]  $P=(P_{ij})$ is an $p\times q$ block matrix with each block $P_{ij}\in \R^{q\times p}$.
  \item[(b)]  For each $i\in [p], j\in [q]$, $P_{ij}=(k_{st}^{(i,j)})$ is a (0,1) matrix with a unique 1 which lies at the position $(j,i)$.     
\end{description}
We denote this commutation matrix by $K_{p,q}$.  Thus a commutation matrix is of size $pq\times pq$. \\
\begin{exm}
$K_{2,3}$ is a $6\times 6$ permutation matrix partitioned as a $2\times 3$ block matrix, i.e.,  
\beq\label{eq:cmex01}
 K_{2,3} =\begin{pmatrix} H_{11} & H_{12} & H_{13} \\  H_{21} & H_{22} & H_{23}\end{pmatrix}
\eeq
where each block $H_{ij}=(h^{(i,j)}_{s,t})$ is a $3\times 2$ matrix whose unique nonzero entry is $h^{(i,j)}_{j,i}=1$.  Specifically 
 \beq\label{ex:01}
 K_{2,3} = \begin{bmatrix}
                  1&0&0&0&0&0\\  
                  0&0&1&0&0&0\\  
                  0&0&0&0&1&0\\ 
                  0&1&0&0&0&0\\
                  0&0&0&1&0&0\\
                  0&0&0&0&0&1\end{bmatrix} 
\eeq 
\end{exm}
\indent  Let $A=K_{p,q}=(A_{st})$ be an $p\times q$ block matrix with block size $p\times q$. For each pair $(i, j)$ with $i\in \dbra{p}, j\in \dbra{q}$, we 
denote $A_{ijkl}$ as the entry at the position $(k,l)$ in the block $A_{ij}$.  We denote $A_{ij}(k,l)$ for the $(i, j)$-entry of block $A_{ij}$. Then 
$A_{ij}(kl)$ is the $(s, t)$-entry of $A$, denoted $a_{s t}$ where 
\beq\label{eq: 4to2}
s = (i-1)q+k, \qquad  t = (j-1)p+l.  
\eeq
Conversely,  given an $(s, t)$th entry of $K_{p,q}$, we can also find its position according to the block form of $A$, i.e., $a_{s t}=A_{k_{1}k_{2}k_{3}k_{4}}$, 
where 
\beyy\label{eq: 2to4}
   k_{1}=\upfit{i/q},                         &k_{2}=\upfit{j/p}. \\
   k_{3}=i-k_{1}q (\texttt{mod}\ q ), & k_{4}= j - k_{2}p (\texttt{mod}\ p). 
\eeyy
\indent The results in the following lemma are some fundamental properties for commutation matrices. 
\begin{lem}\label{le1}
Let $K_{p,q}$ be the commutation matrix. Then we have 
\begin{enumerate}
\item $K_{p,q}^{\top} = K_{q,p}$ and $K_{p,q} K_{q,p}=I_{pq}$. 
\item $K_{p,1}=K_{1,p} = I_{p}$.
\item $\det(K_{p,p})=(-1)^{\frac{p(p-1)}{2}}$.    
\end{enumerate}
\end{lem} 

\begin{proof}
The first two items are obvious. We need only to prove the last item. 

\end{proof}

\indent Let $A\in \R^{m\times n}, B\in \R^{p\times q}$. The Kronecker product of $A,B$, denoted $A\otimes B$, is defined as an $mp\times nq$ matrix in the $m\times n$ 
block form, i.e., $C=A\otimes B=(C_{ij})_{1\le i\le m,1\le j \le n}$ with $C_{ij}=a_{ij}B$.  The following propositions on the Kronecker product of matrices will be used in the 
sequence. 

\begin{lem}\label{le2}
Let $A_{i}\in \R^{m_{i}\times n_{i}}, B_{i}\in \R^{n_{i}\times p_{i}}$ for $i=1,2$.  Then 
\begin{description}
  \item[(1)]  $(A_{1}\otimes A_{2})(B_{1}\otimes B_{2}) = (A_{1}B_{1})\otimes (A_{2}B_{2})$. 
  \item[(2)]  $(A_{1}\otimes A_{2})^{\p} = (A_{1})^{\p} \otimes (A_{2})^{\p}$. 
  \item[(3)]  Let $A,  B$ be both invertible.  Then $A\otimes B$ is invertible with its inverse $(A\otimes B)^{-1}  = A^{-1}\otimes B^{-1}$. 
\end{description}
\end{lem}

\indent  The matrix vectorisation, denoted $\vecc(\dot)$, is to turn a matrix into a column vector by vertically stacking all the columns of the matrix in a nature order.  
More specifically, let $A\in \R^{p\times q}$ and $A_{*j}$ be the $j$th column of $A$.  Then consequence of the \emph{vectorisation} of $A$ is an $pq$-dimensional 
vector $\vecc(A)$ with 
\[ \vecc(A)^{\top}:=[A_{*1}^{\p},A_{*2}^{\p},\ldots, A_{*n}^{\p}] \in \R^{pq} \]
Conversely, a vector $\bx=(x_{1}, x_{2}, \ldots, x_{n} )^{\top}\in \R^{n}$ with length $n=pq$ ($p,q>1$) can always be reshaped (matricized) into an $p\times q$ matrix 
$X$ either by column (i.e., the first $p$ entries of $\bx$ form the first column, the next $p$ entries form the second column, etc.). Similarly we can also matricize vector 
$\bx$ rowisely.  Both can be regarded as an 1-1 correspondence between $\R^{pq}$ and $\R^{p\times q}$.  The elements of the $p\times q$ matrix 
$X=(m_{ij})\in \R^{p\times q}$ obtained from the columnwise matricization is defined by 
\beq\label{eq: vec2matrx01}
 c_{ij} = x_{i+(j-1)p}, \quad   \forall 1\le i\le p, 1\le j\le q
\eeq
and the elements of the $p\times q$ matrix $X=(m_{ij})\in \R^{p\times q}$ obtained from the rowise matricization is defined by 
\beq\label{eq: vec2matrx02}
r_{ij} = x_{j+(i-1)q}, \quad   \forall 1\le i\le p, 1\le j\le q
\eeq
We use $\vecc^{-1}_{p,q}$ to denote for the columnwise matricization of an $pq$-dimensional vector into an $p\times q$ matrix, and use  
$\vecr^{-1}_{p,q}$ to denote the rowise matricization of an $pq$-dimensional vector into an $p\times q$ matrix.\\
\indent The following property, which can be found in many textbook on the matrix theory, is crucial to the multivariate statistical models.  
        
\begin{lem}\label{le3}
Let $A\in \R^{m\times n},  B\in \R^{n\times p},  C\in \R^{p\times q}$.  Then we have   
\beq\label{kroneckp301}
\vecc(ABC)  = (C^{\p}\otimes A) \vecc(B)
\eeq
For $p=m$, we also have  
\beq\label{kroneckp302}
\tr(AB)  = \vecc(B)^{\p}\vecc(A)
\eeq
\end{lem} 

 \indent In the next section, we first present some basic properties of the commutation matrices. Section 3 is dedicated to the commutation tensors where we first present 
 the definition of the commutation tensor and study some of its properties. In Section 4 we employ the commutation tensor to study the linear preserving problem (LPP) 
 and extend the LPP to a more general multilinear preserving problem (MLPP). We will also use the relationship between a matrix and its transpose through the commutation
 matrix (tensor) to unify the form of a linear determinant preserver and linear rank preserver.  

\vskip 5pt

\section{Commutation matrices}
\setcounter{equation}{0}

The following result presents a linear relationship between $\vecc(A^{\top})$ and $\vecc(A)$ through the commutation matrix $K_{p,q}$. 

\begin{thm}\label{th1}
Let $p,q$ be two positive integers. Then 
\beq\label{eq: a2at}
 K_{p,q}\vecc(X)  = \vecc(X^{\top}), \quad \forall X\in \R^{p\times q}
\eeq
Furthermore, $K_{p,q}$ is the unique matrix for (\ref{eq: a2at}) to be held. 
\end{thm}

\begin{proof}
Let $n:=pq$ and $X=(x_{ij})\in \R^{p\times q}$ and denote $\by=K_{p,q}\vecc(X)$. Then $\by\in \R^{n}$. For any $s\in \dbra{n}$, $s$ can be written as 
\[  s = (i-1)q+j, \quad 1\le i\le p, 0\le j<q \]  
that is, $i-1$ and $j$ are respectively the quotient and the remainder of the number $s$ divided by $q$.  Thus by definition,   
\beq\label{eq: prfth1} 
y_{s} = (\sum\limits_{k=1}^{q} H_{ik}X_{\cdot{}k})_{j} = x_{ij}  
\eeq
Here $X_{\cdot{}k}$ denotes the $k$th column of $X$ and $H_{ik}$ is the $(i,k)$th block of $K_{p,q}$ whose unique nonzero entry (equals 1) is at the position 
$(k,i)$ by definition. On the other hand, we have $\vecc(X^{\top})_{s}=a_{ij}$. Thus (\ref{eq: a2at}) holds.\\
\indent  Now we suppose there is a matrix $K\in \R^{n\times n}$ such that $K\vecc(X) =\vecc(X^{\top})$ holds for all $X\in \R^{p\times q}$.  Then 
$(K-K_{p,q}) \vecc(X) =0$ for all $X\in \R^{p\times q}$. It follows that $K-K_{p,q}=0$ and consequently $K=K_{p,q}$. 
\end{proof} 

\indent Theorem \ref{th1} tells us that the transpose of a matrix $A$ can be regarded as the permutation of $A$ through the commutation matrix, but this linear 
transformation is established in terms of the matrix vectorisation, which, nevertheless, alters the shape of the matrix.  But sometimes we do want to know exactly the 
relation of $A$ and its transpose while preserving its shape. This will be done in the next section.\\
\indent  We denote $e_{r, s}$ for the $s$th coordinate vector of $\R^{r}$.  The following lemma can be regarded as the rank-1 decomposition of $K_{p,q}$.  
\begin{lem}\label{le2.1}
\beq\label{eq: decompK}
K_{p,q} = \sum\limits_{i,j} (e_{p,i}\otimes e_{q,j})(e_{q,j}\otimes e_{p,i})^{\top} 
\eeq
Here the summation runs over all $i\in \dbra{p}, j\in \dbra{q}$. 
\end{lem}  

\begin{proof}
We let $F^{(i,j)}=(G_{st})\in \R^{pq\times pq}$ be the $p\times q$ block matrix, each block $G_{st}\in \R^{q\times p}$ is a zero block except the $(i,j)$th block 
$G_{ij}=E_{ji}^{\top}$. Here $E_{ji}=e_{q,j}e_{p,i}^{\top}$ is the elementary matrix with the unique 1 at position $(j,i)$.  Then it is obvious that   
\beq\label{eq: decompK01} 
K_{p,q} = \sum\limits_{i\in\dbra{p}, j\in\dbra{q}} F^{(i,j)} 
\eeq
Note that 
\beq\label{eq: decompK02} 
F^{(i,j)} = e_{p,i}e_{q,j}^{\top}\otimes e_{q,j}e_{p,i}^{\top} = (e_{p,i}\otimes e_{q,j})(e_{q,j}\otimes e_{p,i})^{\top} 
\eeq
The last equality of (\ref{eq: decompK02}) follows directly from the first item of Lemma \ref{le2}, and the decomposition (\ref{eq: decompK}) follows directly from the 
combination of  (\ref{eq: decompK01}) and (\ref{eq: decompK02}).  
\end{proof} 
 
\indent The following result, showing an essential role of the commutation matrix in the linear and multilinear algebra, will be employed in the proof of our main result.  
\begin{thm}\label{th2.2}
Let $A\in \R^{pq\times pq}$ where $p, q >1$ are positive integers. Then  
\beq\label{eq:Kbasic}
A (\bx\otimes \by) = \by\otimes \bx,  \quad  \forall \bx\in \R^{q}, \by\in \R^{p}  
\eeq
if and only if  $A=K_{p,q}$.  
\end{thm}

\begin{proof} 
We first prove the sufficiency.  Let $A=K_{p,q}$. Then for any $\bx\in \R^{q}, \by\in \R^{p}$, by Lemma \ref{le2.1} we have 
\beyy\label{eq: Kbasicprf01} 
K_{p,q} (\bx\otimes \by) &=& (\sum\limits_{i,j} (e_{p,i}\otimes e_{q,j})(e_{q,j}\otimes e_{p,i})^{\top})(\bx\otimes \by)\\
                                      &=& (\sum\limits_{i,j} (e_{q,j}^{\top}\bx) \otimes (e_{p,i}^{\top}\by) (e_{p,i}\otimes e_{q,j})\\ 
                                      &=& (\sum\limits_{i,j} x_{j}y_{i} e_{p,i}\otimes e_{q,j}\\
                                      &=& \by\otimes \bx.  
\eeyy
The third equality is due to Lemma \ref{le2}.  \\
\indent Conversely, we suppose condition (\ref{eq:Kbasic}) holds.  We want to show that $A=K_{p,q}$. For each $i\in \dbra{p}, j \dbra{q}$, we let 
$\bx=e_{j}\in \R^{q}$ be the $j$th coordinate vector of $\R^{q}$.  By the blocking product of $A(\bx\otimes \by)$ and (\ref{eq:Kbasic}), we have 
\beq\label{eq: th2.2prf01} 
A_{ij}\by = y_{i} e_{j}, \quad \forall i,j\in\dbra{p}, \quad  \forall \by\in \R^{p}
\eeq
Denote $E_{ij}\in \R^{q\times p}$ for the fundamental matrix, i.e., all of whose entries are zero except the $(i,j)$ entry which is 1.  
Then we have 
\beq\label{eq: th2.2prf02}
E_{ji}\by =y_{i}e_{j}, \quad  \forall i\in\dbra{p}, j\in \dbra{q}
\eeq
where $\by=(y_{1},y_{2},\ldots, y_{p})^{\top}$.  Thus we have from (\ref{eq: th2.2prf01}) and (\ref{eq: th2.2prf02}) that 
\beq\label{eqK6}
(A_{ij}-E_{ji}) \by = 0, \quad   \forall \by\in \R^{p}
\eeq   
It follows that $A_{ij}=E_{ji}$ for all $i,j$. Consequently we have $A=K_{p, q}$ by the definition. 
\end{proof}

\indent An alternative proof to Theorem \ref{th2.2} is to employ Theorem \ref{th1}:  we denote $X=\bx\by^{\top}\in \R^{p\times q}$. Then 
\[ \bx\otimes \by =\vecc(X), \quad  \by\otimes \bx =\vecc(X^{\top}) =\vecc(\by\bx^{\top}) =\by\otimes \bx \]
By Theorem \ref{th1}, we have 
\[  K_{p,q} (\bx\otimes \by) =K_{p,q} \vecc(X) = \vecc(X^{\top}) = \by\otimes \bx. \]
  
\indent Since $\R^{pq}$ is isometric to $\R^{p}\times \R^{q}$, which is also isometric to $\R^{q}\times \R^{p}$, $K_{p,q}$ can be regarded as 
a \emph{block permutation} on $\R^{pq}$. Consequently it can be regarded as an automorphism on $\R^{n\times n}$ where $n=pq$. The following 
result, which is an improvement of a known property for the commutation matrices, enhances this point.    

\begin{cor}\label{cor2.3}  
Let $A\in \R^{p\times p}, B\in \R^{q\times q}$ and let $n=pq$ where $p,q$ are positive integers.  Then we have  
\beq\label{eq:comm4kron}
A\otimes B =K_{p,q}(B\otimes A)K_{p,q}   
\eeq
Furthermore, if $p=q$,  then $A\otimes B$ is permutation similar to $B\otimes A$.  
\end{cor}

\begin{proof}
Since $K_{p,q}^{\top} = K_{q,p}$ is an orthogonal matrix by the first item of Lemma \ref{le1}, (\ref{eq:comm4kron}) can be equivalently written as 
\beq\label{eq:comm4kron2}
K_{q,p}(A\otimes B) = (B\otimes A)K_{q,p}   
\eeq
For any vector $\bx\in \R^{p}, \by\in \R^{q}$, we have, by Theorem \ref{th2.2},  
\beyy\label{eq: Kthm23left} 
K_{q,p}(A\otimes B)(\bx\otimes \by) &=& K_{q,p} (A\bx\otimes B\by)\\ 
                                                         &=& B\by\otimes A\bx\\   
                                                         &=& (B\otimes A)(\by\otimes\bx)\\
                                                         &=& (B\otimes A)K_{q,p}(\bx\otimes \by)
\eeyy
Thus (\ref{eq:comm4kron2}) holds.\\  
\indent Now if $p=q$, then the commutation matrix $K_{p}=K_{p,p}=K_{p}^{\top}$ by Lemma \ref{le1}. Thus by (\ref{eq:comm4kron}) we know that $A\otimes B$ is permutation similar to $B\otimes A$.  
\end{proof}

\indent For $p=q$, we denote $n=p^{2}$ and $K_{p}:=K_{p,p}$. Then we have 
\begin{thm}\label{th: 2.4}
\begin{enumerate}
\item $K_{p}$ is a symmetric involution, i.e.,  $K_{p}^{2}=I_{n}$. 
\item $\tr(K_{p})=p$. 
\item $\det(K_{p})=(-1)^{\frac{p(p-1)}{2}}$ for any integer $p>1$. 
\end{enumerate}
\end{thm}
 
\begin{proof} 
\begin{description}
  \item[(1). ]  The symmetry of $K_{p}$ follows directly from  (1) of Lemma \ref{le1}, and (1-2) of Lemma  \ref{le1} yields the equation $K_{p}^{2}=I_{n}$.
  \item[(2). ]  This is obvious since $\tr(K_{p})=\sum\limits_{i=1}^{p} \tr(E_{ii})=p$ where $E_{ii}\in \R^{p\times p}$ whose entries are all zeros except the $(i,i)$-entry 
  that is 1.    
  \item[(3). ] It is immediate from (1) that $\abs{\det(K_{p})}=1$.  By (1), $K_{p}$ is an orthogonal projection, thus $K_{p}$ can be decomposed as $K_{p}=UDU^{\top}$,  
where $U\in \R^{n\times n}$ is an orthogonal matrix and $D=diag(I_{r}, -I_{n-r})$. By (2), we have $p =\tr(K_{p}) = r-(n-r) = 2r-n$. It follows that $r=\frac{1}{2}p(p+1)$. 
Thus we have 
\[ \det(K_{p}) =(-1)^{ p^{2}-r} = (-1)^{\frac{1}{2}p(p-1)} \]
Consequently we obtain (3). 
\end{description}
\end{proof}

\vskip 5pt
\section{Commutation Tensors}
\setcounter{equation}{0}

\indent  In this section, we define the commutation tensor and investigate its properties. We use the commutation tensor to obtain an unified form of the linear rank preserver. \\

\indent  Recall that an $m$-order tensor $\A=(A_{i_{1}i_{2}\ldots i_{m}})$ of size $N_{1}\times N_{2}\times \ldots \times N_{m}$ can be regarded 
as an $m$-way array where the subscripts $(i_{1},i_{2},\ldots, i_{m})$ is taken from the set 
\[ S(m,n):=\set{\si=(i_{1},i_{2}, \ldots, i_{m}): i_{k}\in\dbra{N_{k}},  k=1,2,\ldots, m. }  \]
For $N_{1}=N_{2}=\ldots =N_{m}=N$, we call $\A$ an $(m,n)$-tensor, and denote $\T_{m,n}$ for the set of all \mnrts.  
An $(m,n)$-tensor $\A$ is called symmetric if for any $\si=(i_{1},i_{2}, \ldots, i_{m})\in S(m,n)$, we have $A_{i_{1}i_{2}\ldots i_{m}}=A_{j_{1}j_{2}\ldots j_{m}}$ 
where $(j_{1},j_{2}, \ldots, j_{m})\in S(m,n)$ is a permutation of $\si$.  We denote $\ST_{m,n}$ for the set of all \mnsts. \\

\indent Given any vector $\bx=(x_{1},x_{2},\ldots, x_{n})^{\top}\in \R^{n}$. We generate a rank-1 $(m, n)$-tensor $\X=(X_{i_{1}i_{2}\ldots x_{m}}):=\bx^{m}$ 
where 
\[ X_{i_{1}i_{2}\ldots x_{m}} = x_{i_{1}}x_{i_{2}}\ldots x_{i_{m}}  \]
More generally, a rank-1 $m$-order tensor is in form $\al_{1}\times \al_{2}\times \ldots \times \al_{m}$ where $\al_{k}\in \R^{N_{k}}$ for $k\in \dbra{m}$. 
An $(m,n)$-tensor $\A = (A_{i_1i_2\ldots i_m})$ is called \emph{positive semidefinite} if  for each $x=(x_1,x_2,\ldots,x_n)^T\in \R^n$  
\beq\label{def: psd}
f(\bx):=\sum\limits_{i_1,i_2,\ldots,i_m} A_{i_1i_2\ldots i_m}x_{i_1}x_{i_2}\ldots x_{i_m} \ge 0
\eeq 
and called \emph{positive definite} if  $f(\bx)>0$ for all $\bx\neq 0$.  It is easy to see that a nonzero positive (semi-)definite tensor must be of an 
even order.\\ 
\indent Let $\A = (A_{i_1i_2\ldots i_m})$ be a tensor of size $I_{1}\times I_{2}\times \ldots \times I_{m}$ and $B\in \R^{I_{n}\times J_{n}}$, $\bx\in \R^{I_{n}}$. 
We define the tensor-vector multiplication along mode-$n$ by 
\[ \A\times_{n} \bx :=\sum\limits_{i_n=1}^{I_{n}} A_{i_1\ldots i_n\ldots i_m}x_{i_n} \]
which produces a $(m-1)$-order tensor.  This definition can also be extended to the multiplication of any two tensors with some consistent dimensions. 
For example, a tensor-matrix multiplication $\A\times_{3,4}B$ along mode-$\set{3,4}$ is defined as
\beq\label{eq: tmprod} 
(\A\times_{3,4}B)_{ij}=\sum\limits_{k,l} A_{ijkl}b_{kl}
\eeq
where $\A\in \R^{m\times n\times p\times q}, B\in \R^{p\times q}$. Then $\A\times_{3,4}B\in \R^{m\times n}$.  As a matrix can be vectorised into a vector, a 
tensor can be flattened or unfolded into a matrix \cite{kolda2009}. \\ 
 
\indent  Let $\A\in \T_{m,n}$ and $\bx=(x_{1},\ldots, x_{n})^{\top}\in \C^{n}$ be a nonzero vector.  $\bx$ is called an eigenvector of $\A$ if there exists a scalar 
$\la\in \C$ such that  
\[ \A\bx^{m-1} =\la\bx^{m-1} \]
If $\bx\in \R^{n}$, then  $\la\in \R$. We call such a $\la$ an H-eigenvalue of $\A$ and $\bx$ the eigenvector of $\A$ corresponding  to $\la$. \\ 

\indent Given any positive integer $m, n$. We define the $(m,n)$-\emph{commutation tensor} $\K_{n,m}=(K_{ijkl})$ to be an (0,1) $n\times m\times m\times n$ tensor 
where  $K_{ijkl}=1$ if and only if $i=l, j=k$ for all $1\le i, l \le n, 1\le j,k\le m$.  Note that $\K_{n,m}$ can be flattened into the commutation matrix $K_{m,n}$ and that 
$K(i,j,:,:)=H_{ij}\in \R^{m\times n}$ for all $1\le i\le n, 1\le j\le m$.  $\K_{m,n}$ is also called a permutation tensor.   
\begin{exm}\label{exm301}
Consider $\K_{2,3}=(K_{ijkl})$ with size $2\times 3\times 3\times 2$. Then $\K$ has six nonzero entries as 
\[ K_{1111}=K_{2112}=K_{1221}= K_{2222}=K_{1331}=K_{2332}=1 . \]  
Given any matrix $X=(x_{ij})\in \R^{3\times 2}$.   It is easy to see that $\K\times_{3,4}X =X^{\top}$. 
\end{exm}

\indent An  even-order tensor $A=(A_{i_{1}\ldots i_{m}j_{1}\ldots j_{m}})$ is called \emph{pair-symmetric} if  
\beq
A_{i_{\tau(1)} \ldots i_{\tau(m)} j_{\tau(1)}\ldots j_{\tau(m)}} = A_{i_{1}\ldots i_{m}j_{1}\ldots j_{m}} 
\eeq 
where $\tau\in S_{m}$ is an arbitrary permutation. Pair-symmetric tensors have applications in elastic physics\cite{HQ2018}. Obvious that  $\K=\K_{n,n}$ is pair-symmetric. \\
\indent Now we can establish a multi-linear relationship between a matrix and its transpose through the commutation tensor.   

\begin{thm}\label{th:ctensor02}
Let $\K=\K_{m,n}$ be an $(m,n)$-commutation tensor. Then $\K\times_{3,4} X=X^{\top}$ for any $m\times n$ matrix $X$. 
\end{thm}

\begin{proof}
Denote $B:=\K\times_{\set{3,4}} X =(B_{ij})$. Then $B\in \R^{m\times n}$. By definition (\ref{eq: tmprod}) we have for each $i\in \dbra{m}, j\in [n]$
\[ B_{ij} = \sum\limits_{k, l} K_{ijkl}X_{kl} = K_{ijji}X_{ji} = X_{ji} \]
Thus $B=X^{\top}$. 
\end{proof} 
\indent  We denote the multiplication $\A\times_{\set{3,4}} B$ simply by $\A B$ in the following for our convenience. The definition can also be extended to the case 
for any two tensors $\A,\B$ of order $2m$, where $\A=(A_{i_{1}\ldots i_{m}, j_{1}\ldots j_{m}}), \B=(B_{i_{1}\ldots i_{m}, j_{1}\ldots j_{m}})$ are consistent, i.e., 
\beq\label{eq:defprodin2pn}
(\A\B)_{i_{1}\ldots i_{m}, j_{1}\ldots j_{m}} = \sum\limits_{k_{1},\ldots,k_{m}} A_{i_{1}\ldots i_{m}, k_{1}\ldots k_{m}} B_{k_{1}\ldots k_{m}, j_{1}\ldots j_{m}} 
\eeq
or to the case where $\A$ is of order $2m$ and $\B$ is of order $m$ in a similar way, resulting in an $m$-order tensor.\\
\indent  Let $\A,\B\in \T_{2m,n}$ whose multiplication is defined by (\ref{eq:defprodin2pn}). Then $\T_{2m,n}$ is closed under the multiplication.  Furthermore  
\begin{lem}\label{le:t4assoc}
$\T_{2m,n}$ obeys an associative law under the multiplication defined by (\ref{eq:defprodin2pn}), i.e.,  
\beq\label{eq:assocT2pn}
(\A\times \B)\times \C =  \A\times (\B\times \C)
\eeq 
\end{lem}

\begin{proof}
Denote $\F =\A\times \B=(F_{i_{1}\ldots i_{m}, j_{1}\ldots j_{m}}), \caH =\B\times \C=(H_{i_{1}\ldots i_{m}, j_{1}\ldots j_{m}})$,  and 
$\G=\F\times \C =(G_{i_{1}\ldots i_{m};j_{1}\ldots j_{m}})$. For convenience, we denote $\si^{(i)}:=(i_{1}\ldots i_{m})$. Thus 
\[(\si^{(i)},\si^{(j)}) = (i_{1}\ldots i_{m}, j_{1}\ldots j_{m})\in S(2m, n)\]
and $\F =\A\times \B=(F_{\si^{(i)}, \si^{(j)}})$. We have by definition  
\beyy 
 G_{\si^{(i)}, \si^{(j)}}  &=& \sum\limits_{\si^{(k)}} F_{\si^{(i)},\si^{(k)}} C_{\si^{(k)}, \si^{(j)}}\\
              &=& \sum\limits_{\si^{(k)}} (\sum\limits_{\si^{(l)}} A_{\si^{(i)}, \si^{(l)}} B_{\si^{(l)},\si^{(k)}}) C_{\si^{(k)},\si^{(j)}} \\
              &=& \sum\limits_{\si^{(l)}} A_{\si^{(i)}, \si^{(l)}} (\sum\limits_{\si^{(k)}} B_{\si^{(l)},\si^{(k)}} C_{\si^{(k)},\si^{(j)}}) \\
              &=& \sum\limits_{\si^{(l)}} A_{\si^{(i)}, \si^{(l)}} H_{\si^{(l)}, \si^{(j)}} 
\eeyy
The right hand side of the last equality is exactly the entry $[\A\times (\B\times \C)]_{\si^{(i)}, \si^{(j)}}$. Thus we complete the proof of  (\ref{eq:assocT2pn}). 
\end{proof}

\indent The equality (\ref{eq:assocT2pn}) in Lemma \ref{le:t4assoc} holds for any $2m$-order tensors $\A, \B, \C$ whenever the multiplications in (\ref{eq:assocT2pn}) 
make sense. Now we denote $\K=\K_{n,n}$ when no risk of confusion arises for $n$, and $\K^{2}=\K\times \K$ is called the \emph{square of} $\K$.  We may also define 
any power $\K^{m}$ recursively due to Lemma \ref{le:t4assoc} for any positive integer $m$, i.e.,  $\K^{m}=\K^{m-1}\K$.  Note that the definition of a tensor power can be
extended to any even-order tensor. 

\begin{thm}\label{th:ctensor03}
\begin{enumerate}
\item $\K^{m}=\K$ for any odd number $m$. 
\item $\K^{m}=\K^{2}$ for any positive even number $m$. 
\end{enumerate}
\end{thm}
 
\begin{proof}
(1). For each $(i,j,k,l):  i,l\in \dbra{m}, j,k\in [n]$, we have by definition
\beq\label{eq: k2}
K^{2}_{ijkl}= \sum\limits_{i_{1}, j_{1}}  K_{iji_{1}j_{1}}K_{i_{1}j_{1}kl} = K_{ijji}K_{jikl} =K_{jikl}
\eeq
It follows that $K^{2}_{ijkl}=1$ if and only if $i=k, j=l$.  Thus 
\beq\label{eq: k3to1} 
\K^{3}_{ijkl} = \sum\limits_{k^{\p},l^{\p}}  K^{2}_{ijk^{\p}l^{\p}} K_{k^{\p}l^{\p}kl} = K^{2}_{ijij} K_{ijkl} 
\eeq
for each $(i,j,k,l):  i,l\in \dbra{m}, j,k\in [n]$, which implies $\K^{3}=\K$. Thus $\K^{m}=\K$ for any odd number $m$ if we apply recursively the fact (\ref{eq: k3to1}).\\
\indent To prove the second item, we note that from $\K^{3}=\K$ it follows that $\K^{4}=\K^{2}$ by the associative law (i.e., Lemma \ref{le:t4assoc}). Consequently 
for any even $m=2k$ ($k\ge 2$) we have by recursion that $\K^{m}=\K^{2}$.  We note that item (2) can also be proved as in the follows.
\beyy\label{eq: 2to4} 
K^{4}_{ijkl}  &=& \sum\limits_{i_{1}, j_{1}}  K^{2}_{iji_{1}j_{1}}K^{2}_{i_{1}j_{1}kl}\\
                   &=& \sum\limits_{i_{1}, j_{1}}  K_{j i i_{1}j_{1}}K_{j_{1}i_{1}kl} =K_{jikl} 
\eeyy
where the second equality is due to (\ref{eq: k2}). Thus we have $\K^{4}=\K$.  This in turn follows by $\K^{6}=\K^{2}\K^{4}=\K^{3}=\K$. Consequently, 
we get (2). 
\end{proof} 

 \indent  Now we let $\B=(B_{ijk})\in \R^{n\times n\times m}$ with $B(:,:,k)=B_{k}\in \R^{n\times n}$ corresponding to a permutation $\pi_{k}\in S_{n}$.
Define $\K^{\pi}:=(K_{i_{1}\ldots i_{m}; j_{1}\ldots j_{m}})=B_{1}\times B_{2}\times \ldots \times B_{m}$ with $\pi=(\pi_{1},\ldots, \pi_{m})$ and  
\[ K_{i_{1}\ldots i_{m}; j_{1}\ldots j_{m}} = B_{i_{1}j_{1} 1} B_{i_{2}j_{2} 2}\ldots B_{i_{m}j_{m} m}  \]
$\K^{\pi}$ is called a \emph{generalised commutation tensor associated with} $\pi$, or briefly a \textbf{$\pi$-GCT}.  By definition, we have  
 \beq\label{eq: defgct}
   K_{i_{1}\ldots i_{m};j_{1}\ldots j_{m}} =1 \  \ffi    j_{k} =i_{\tau(k)} , \quad  \forall k, (i_{1},\ldots, i_{m})\in S(m,n)
 \eeq
 This is equivalent to 
 \[  \K^{\si} = \overbrace{B\times B\times\ldots\times B}^{m} \]
where $B\in \R^{n\times n}$ is a permutation matrix corresponding to $\si$.  Note that a GCT $K$ becomes a commutation tensor if $m=2$.\\
 \indent  Let $\si\in S_{n}, \A\in \T_{m,n}$. We will show that tensor $\A\K^{\pi}$ (also $\K^{\pi}\A$) can be regarded as a permutation of $\A$ 
whose specific meaning is described in the following.  For this purpose, we define tensor $\A^{\si}=(\tA_{i_{1}\ldots i_{m}})$ by 
\[ \tA_{i_{1}i_{2}\ldots i_{m}} = A_{\si(i_{1}) \si(i_{2})\ldots \si(i_{m})}, \quad  \forall  (i_{1},\ldots, i_{m})\in S(m, n) \]
It is shown by Comon et al. \cite{cglm2008} that an $m$th order $n$ dimensional tensor $\A\in \T_{m,n}$ can always be decomposed into
\beq\label{eq: rk1decomp}
\A = \sum\limits_{j=1}^{R} \al_{1j}\times \al_{2j}\times \ldots \times \al_{mj}
\eeq
where each $\al_{ij}\in \R^{n}$ is nonzero. The formula (\ref{eq: rk1decomp}) is called a \emph{rank-1 decomposition} or a \emph{CP decomposition} of 
$\A$, the smallest possible number $R$ in (\ref{eq: rk1decomp})  is called the \emph{rank} of $\A$. Furthermore, $\A$ can be decomposed into form 
\beq\label{eq: symrk1decomp}
\A = \sum\limits_{j=1}^{R} \al_{j}^{m}
\eeq
if $\A$ is a symmetric tensor, where $0\neq \al_{j}\in \R^{n}$.  Given a tensor $\A\in \T_{m,n}$ and a matrix $B\in \R^{n\times n}$. We define the 
\emph{right complete product} of $\A$ by $B$, denoted $\A\cdot{}\dbra{B}$, by 
\[ \A\cdot{}\dbra{B}:= \A\times_{1}B\times_{2}B\times_{3}\ldots \times_{m}B.  \]
The \emph{complete left product} of $\A$ by $B$,denoted $\dbra{B}\cdot{}\A$, is defined analogically. Note that for any $k\in \dbra{m}$, we have 
\[ (\A\times_{k}B)_{i_{1}i_{2}\ldots i_{m}} = \sum\limits_{i=1}^{n} A_{i_{1}\ldots i_{k-1} i i_{k+1}\ldots i_{m}} B_{i i_{k}} \]
and 
\[ (B\times_{k}\A)_{i_{1}i_{2}\ldots i_{m}} = \sum\limits_{j=1}^{n} B_{i_{k} j} A_{i_{1}\ldots i_{k-1} j i_{k+1}\ldots i_{m}} \]  
\indent  Now suppose $\A$ has a CP decomposition (\ref{eq: rk1decomp}) ( $0\neq \al_{ij}\in \R^{n}$).  Then $\A^{\si}$ can also be written as 
\beq\label{eq: pirk1dec}
\A^{\si}: = \sum\limits_{j=1}^{R} \al_{\si(1),j}\times \al_{\si(2),j}\times \ldots \times \al_{\si(m),j}
\eeq
For $m=2$, $A^{\tau}$ is either $A$ itself (when $\tau$ is the identity map) or the transpose of $A$ (when $\tau=(12)$ is a swap). 

\begin{lem}\label{le: 403}
Let $\A\in \T_{m,n}$ be a real tensor with a CP decomposition (\ref{eq: rk1decomp}) and let $\tau\in S_{m}$. Then   
\beq\label{eq: K4Api}
\A^{\tau} = \A\times \K^{\tau}  
\eeq  
where $\K^{\tau}\in \T_{2m; n}$ is an $2m$-order GCT defined by (\ref{eq: defgct}).  
\end{lem}

\begin{proof}
For any $(i_{1},i_{2},\ldots,i_{m})\in S(m,n)$, we let $\tau(i_{1}, i_{2},\ldots, i_{m})=(j_{1}, j_{2},\ldots, j_{m})$, i.e., $j_{k}=i_{\tau(k)}$ for all $k\in \dbra{m}$.
Then by (\ref{eq: rk1decomp}) 
\beyy\label{eq: Ataudec}
(A^{\tau})_{i_{1}i_{2}\ldots i_{m}} &=& \sum\limits_{j=1}^{R} a^{\tau(1)}_{i_{1} j} a^{\tau(2)}_{i_{2} j}\ldots a^{\tau(m)}_{i_{m} j} \\
                                                    &=& \sum\limits_{j=1}^{R} a^{(1)}_{j_{1} j} a^{(2)}_{j_{2} j}\ldots a^{(m)}_{j_{m} j}\\
                                                    &=& A_{j_{1} j_{2}\ldots j_{m}}
\eeyy
On the other side
\beyy\label{eq: Ataudecright}
(\A\times \K^{\tau})_{i_{1}i_{2}\ldots i_{m}} &=& \sum\limits_{k_{1},k_{2},\ldots, k_{m}} A_{k_{1}\ldots k_{m}} K_{k_{1}\ldots k_{m}; i_{1}\ldots i_{m}}\\
                                                                   &=& A_{j_{1} j_{2}\ldots j_{m}}
\eeyy
Thus $\A^{\tau} = \A\times \K^{\tau}$ for any $\tau\in S_{m}$.  The proof is completed.
\end{proof}

\indent We now denote $\K(\pi,p) =\overbrace{B\times\ldots \times B}^{p}$, where $B$ is the permutation matrix associated with $\pi\in S_{n}, p\in \dbra{m}$. Thus 
 $\K(\pi,m) =\K^{\pi}$. We denote 
 \[ \KK(m,n):=\set{\K^{\pi, m}: \quad  \pi\in S_{n} }, \]
The multiplication on $\KK(m,n)$ can be defined by (\ref{eq:defprodin2pn}). Then we have 
\begin{thm}\label{th:tpergroup}   
\begin{description}
  \item[(1)]  $\KK(m,n)$ is a subgroup of $\T_{2m,n}$ under the tensor multiplication defined by (\ref{eq:defprodin2pn}).
  \item[(2)]  $\K^{(id)}:=\K^{(id,m)} =\overbrace{I_{n}\times\ldots \times I_{n}}^{m}$ (corresponding to the identity map) is the unique identity element in $\T_{2m,n}$, i.e., 
\beq\label{eq: idel}
\X\times \K^{(id)} = \K^{(id)} \times \X =\X, \quad  \forall \X \in \T_{2m,n}
\eeq
 \item[(3)]  Every element $\K^{\pi}\in \KK(m,n)$ is invertible. Furthermore, its inverse is $(\K^{\pi})^{-1} = \K^{\pi^{-1}}$. 
\end{description}
\end{thm}

\begin{proof} 
To prove (1), we let $\K_{i}=\overbrace{P_{i}\times \ldots \times P_{i}}^{m}$ where $P_{i}\in \R^{n\times n}$ corresponds resp. to a permutation $\pi_{i}\in S_{n}$, $i=1,2$. 
Then we have 
\beyy 
    \K_{1}\K_{2}  &=& (P_{1}\times P_{1}\times \ldots \times P_{1}) (P_{2}\times P_{2}\times \ldots \times P_{2}) \\
                         &=& (P_{1}P_{2})\times (P_{1}P_{2})\times \ldots \times (P_{1}P_{2}) \\
                         &=&  P\times P\times \ldots \times P 
\eeyy
where $P=P_{1}P_{2}\in \R^{n\times n}$ is also a permutation matrix corresponding to a permutation $\pi=\pi_{1}\pi_{2}$. \\
\indent  For (2), it suffices to prove the equality $\K^{(id)} \times \X =\X$ in (\ref{eq: idel}) since the other part is similar to it. Denote $\Y = \X\times \K^{id}$. Then 
\beyy\label{eq: identity} 
Y_{i_{1}\ldots i_{m};j_{1}\ldots j_{m}}  &=& \sum\limits_{k_{1},\ldots, k_{m}}  X_{i_{1}\ldots i_{m};k_{1}\ldots k_{m}} K_{k_{1}\ldots k_{m};j_{1} \ldots j_{m}}\\
                                                      &=& \sum\limits_{k_{1},\ldots, k_{m}}  X_{i_{1}\ldots i_{m};k_{1}\ldots k_{m}} \delta_{k_{1}j_{1}}\ldots \delta_{k_{m}j_{m}}\\
                                                      &=& \sum\limits_{k_{1},\ldots, k_{m}}  X_{i_{1}\ldots i_{m};j_{1}\ldots j_{m}}       
\eeyy
for all possible $(i_{1},\ldots, i_{m}), (j_{1},\ldots, j_{m})\in S(m,n)$. Thus we have $\X\times \K^{id} =\X$. The similar argument can be applied to prove 
$\X=\K^{id}\times \X$. \\
\indent  For (3), it suffices to verify that  
\beq\label{eq: inv1}
\K^{\pi} \K^{\pi^{-1}} = \K^{id}
\eeq  
and 
\beq\label{eq: inv2}
\K^{\pi^{-1}}\K^{\pi} =\K^{id}
\eeq
To prove (\ref{eq: inv1}), we write $\si:=\pi^{-1}$. Then each entry of the left hand side of (\ref{eq: inv1}) is 
\beyy\label{eq: prfinv1} 
(\K^{\pi} \K^{\si})_{i_{1}\ldots i_{m};j_{1}\ldots j_{m}}  
&=& \sum\limits_{k_{1},\ldots, k_{m}}  K^{\pi}_{i_{1}\ldots i_{m};k_{1} \ldots k_{m}} K^{\si}_{k_{1} \ldots k_{m};j_{1}\ldots j_{m}} \\
&=& \delta_{\si(\pi(i_{1})) j_{1}} \delta_{\si(\pi(i_{2})) j_{2}} \ldots \delta_{\si(\pi(i_{m})) j_{m}} \\
&=& \delta_{i_{1}j_{1}} \delta_{i_{2}j_{2}} \ldots \delta_{i_{m}j_{m}}         
\eeyy
the last equality is due to the fact that  $\si(\pi(k))= (\pi^{-1}\pi)(k)=k$ for each $k\in \dbra{n}$.  It follows that each entry of  the tensor $\K^{\pi} \K^{\si}$ is either 
1 or 0, and that  $(\K^{\pi} \K^{\si})_{i_{1}\ldots i_{m};j_{1}\ldots j_{m}}=1$ if and only if $i_{1}=j_{1}, i_{2}=j_{2}, \ldots, i_{m}=j_{m}$. Thus (\ref{eq: inv1}) holds. 
Similar argument applies to (\ref{eq: inv2}).  The proof is completed.   
\end{proof}

\indent  From Theorem \ref{th:tpergroup}, we can see that $\K^{id}$ is the unique identity element in the group $\T_{2m,n}$.  Given any element $\A\in \T_{2m,n}$.
The invertibility of $\A$ can also be defined as in Item (3) of Theorem \ref{th:tpergroup}, i.e., $\A$ is invertible if there exists a tensor $\B\in \T_{2m,n}$ such that 
\beq\label{eq: inv0}
\A \B = \B \A = \K^{id}
\eeq 
The invertibility of an arbitrary tensor in $\T_{2m,n}$ is too complicated.  But if we consider the following set 
\[ \G_{m,n} :=\set{\A:=\overbrace{A\times\ldots \times A}^{m}:  \forall A\in\R^{n\times n} }  \]
Then $\G_{m,n}\subset \T_{2m,n}$ is isometric to $\R^{n\times n}$.  Furthermore, we denote 
\[ A^{\times m}:= \overbrace{A\times\ldots \times A}^{m} \in \T_{2m,n} \]
and 
\[ \GL_{m,n} :=\set{ A^{\times m}:  A\in GL(n) }  \]
where $GL(n)$ is the set of all the nonsingular matrices in $\R^{n\times n}$.  Then we have 

\begin{thm}\label{th:invertt}
Given any tensor $\A =A^{\times m} \in \G_{m,n}$. Then $\A$ is invertible if and only if  $A$ is invertible. Furthermore, the inverse of $\A$ is 
$(A^{-1})^{\times m}$. 
\end{thm}
\begin{proof}
This result is immediate from the fact that 
\[ 
A^{\times m} B^{\times m} = (\overbrace{A\times\ldots \times A}^{m})(\overbrace{B\times\ldots \times B}^{m}) =\overbrace{AB\times\ldots \times AB}^{m} 
\]
Thus $A^{\times m} B^{\times m} = \K^{id}$ if and only if $AB=I_{n}$, i.e., $B=A^{-1}$. 
\end{proof}
 
\indent Let $w:=\set{i_{1}, i_{2},\ldots, i_{p}}$ be a subset of $\dbra{m}$ ($1\le p\le m$).  The \emph{incomplete product} $\K^{\tau,p}\times_{w}\A$ 
can be interpreted as the simultaneous row permutations of $A[n]$ by $\tau$ for each $n\in w$, where $A[n]$ stands for the matrix 
obtained from the unfolding of $\A$ along mode-$n$.\\     

\indent  For $m=2$, there are two commutation tensors $\K^{(12)}$ and $\K^{id}$, where $\tau=id$ is the identity map on $\set{1,2}$ and $\K:=\K^{(12)}$ is 
the commutation tensor $\K_{n,n}$. It is obvious that $\K^{id}=\K^{2}:=\caI$. \\

\indent  Given a tensor $\A\in \T_{2m,n}$. We call a matrix $A=(a_{ij})$ a \emph{balance unfolding} of $\A$,  if  
\beq\label{eq: defbuf} 
a_{ij} =A_{i_{1}i_{2}\ldots i_{m}; j_{1}j_{2}\ldots j_{m}} 
\eeq
with $i = 1+\sum\limits_{k=1}^{m} (i_{k}-1)n^{k-1}, j = 1+\sum\limits_{k=1}^{m} (j_{k}-1)n^{k-1}$, i.e., each row index $(i_{1},i_{2}, \ldots, i_{m})$ is turned into 
a row index, and each column index $(i_{1},i_{2}, \ldots, i_{m})$ is turned into a column index of $A$(See e.g. \cite{kolda2009}). It is obvious that the balance 
unfolding of a tensor $\A\in \T_{2m,n}$ will produce a matrix of size $n^{m}\times n^{m}$.  A tensor $\A\in \T_{2m,n}$ is called a \emph{balanced permutation tensor} 
(abbrev. BPT) if the consequence of the balance unfolding of $\A$ is a permutation matrix.  We conclude this section by the following property of a nonnegative 
tensor in $\T_{2m,n}$. \\ 
\begin{thm}\label{th: nntinv}
Let $m,n>1$ be two positive integers, and let $\A\in\T_{2m,n}$ be an entrywise nonnegative tensor.  If it has a nonnegative inverse, then the balance unfolding of  
$\A$ is a generalised permutation matrix.    
\end{thm}

\begin{proof}
Let $A$ and $B$ be respectively the balance unfolding of $\A$ and $\B:=\A^{-1}$ where $\A^{-1}$ stands for the inverse of $\A$.  Then 
$A, B\in \R^{n^{m}\times n^{m}}$ are both entrywise nonnegative by the hypothesis.  Since $\A\B=\B\A=\K^{id}$, it follows that 
\[ AB=BA=I_{N},  N:=n^{m} \]
By \cite{Minc1988}, $A$ must be a generalised permutation matrix $A=(A_{ij})\in \R^{n^{m}\times n^{m}}$, that is, for each $i,j\in \dbra{n^{m}}$, there exists a unique 
nonzero positive entry.     
\end{proof}
   
\section{From linear preservers to multilinear preservers} 
\setcounter{equation}{0}

\indent In this section, we will use the commutation tensors to deal with the linear preserving problem (LPP). The linear preserving problem has been investigated 
since the late 20th century by G. Frobenius.  There are a lot of work concerning the LPP. We refer the reader to \cite{} for more detail.\\
\indent A linear map $\T$ on $\C^{n\times n}$ is called a \emph{determinant preserver}  if 
\beq\label{eq: detpres01}
\det(\T(A))=\det(A) \quad  \forall A\in \C^{n\times n} 
\eeq  
 A linear transformation $\T\in (\C^{n\times n})^{\star}$ (which is called the dual space of $\C^{n\times n}$) is called a \emph{rank-1 preserver} 
 if $\rank(A)=1$ always implies $\rank(\T(A)) =1$.  Marcus and Moyls showed in 1959 that 
\begin{lem}\label{le: rk1pre02}
$\T$ is a rank-1 preserver on $\C^{n\times n}$ if and only if there exist invertible matrices $P,Q\in \C^{n\times n}$ such that either 
\beq\label{eq: pres01}
\T(A)=PAQ \quad \forall A\in \C^{n\times n} 
\eeq
or 
\beq\label{eq: pres02}
\T(A)=PA^{\top}Q \quad  \forall A\in \C^{n\times n} 
\eeq  
\end{lem}

\indent In 1977, H. Minc showed \cite{Minc1977} that a linear map $\T$ on $\C^{n\times n}$ is a determinant preserver if and only if there exist invertible matrices 
$P,Q\in \C^{n\times n}$ with $\det(PQ)=1$ such that either (\ref{eq: pres01}) or (\ref{eq: pres02}) holds. A linear rank preserver is surely a rank-1 preserver, the 
inverse is also true when the map is invertible.  A linear determinant preserver must be a rank-1 preserver (by Lemma \ref{le: rk1pre02}).  
As the operations on matrices, there are only two kinds, i.e., elementary operations and the transpose. Our aim is to unify them into one formula.\\
\indent By Theorem \ref{th:ctensor02} we know that the transpose of a matrix $A\in \R^{n\times n}$ is associated with $A$ by commutation tensor $\K_{n}$ through 
$A^{\top} = \K A$. \\ 
\indent  We denote by $\aut(V)$ for the set of all the linear automorphisms on a linear space $V$, and let $\pi\in \aut(\R^{n})$. Then $\pi$ is determined by 
its behaviour on the coordinate vectors $\beps_{1}, \beps_{2}, \ldots, \beps_{n}$, where $\beps_{i}$ is the $i$th coordinate vector of $\R^{n}$. Denote 
$\hat{\beps}_{i}:=\pi(\beps_{i})$ for all $i\in [n]$.  By the linearity of $\pi$, we have 
\beq\label{eq:linpi}
\pi(\bx) = \sum\limits_{j=1}^{n} x_{j} \hat{\beps}_{j},  \quad   \bx=(x_{1},x_{2},\ldots, x_{n})^{\top}\in \R^{n}  
\eeq

\indent We are now ready to describe a linear symmetric automorphism $\phi_{\pi} \in \aut(\ST_{m,n})$. Given any symmetric tensor $\A\in \ST_{m,n}$. 
$\A$ has a symmetric CP decomposition (\ref{eq: symrk1decomp}). Thus     
\beq\label{eq: phi2rk1decomp}
\phi_{\pi}(\A) = \sum\limits_{j=1}^{R} \hal_{j}^{m}
\eeq 
where $\hal_{j}=\pi(\al_{j})$ for each $j\in [R]$.  We call $\phi$ a \emph{positive map} if it preserves the nonnegativity of tensors.  We have 

\begin{lem}\label{le: 301}
Let $\phi\in \aut(\ST_{m,n})$ be a linear symmetric rank-1 preserver.  Then for any $0\neq \bx\in \R^{n}$, there exists a nonzero scalar $\la\in \R$ and 
$0\neq \by\in \R^{n}$ such that 
\beq\label{eq: symrk1}
\phi (\bx^{m}) = \la \by^{m}
\eeq 
\end{lem} 

\begin{proof}
Given any $0\neq \bx\in \R^{n}$. Since $\phi$ is a rank-1 preserver, there exist nonzero vectors, say $\al_{1},\al_{2}, \ldots, \al_{m}\in \R^{n}$, such that 
\beq\label{eq: prfsymrk101}
\phi (\bx^{m}) = \al_{1}\times \al_{2}\times \ldots \times \al_{m} 
\eeq 
Denote $\A=(A_{i_{1}i_{2}\ldots i_{m}})= \al_{1}\times \al_{2}\times \ldots \times \al_{m}$, and $A=[\al_{1},\al_{2},\ldots,\al_{m}]=(a_{ij})\in \R^{n\times m}$. 
The result is equivalent to $\rank(A)=1$. We first note that $\A$ is a nonzero tensor since $\rank(\A)=1$, and thus each $\al_{j}\neq 0$. We may assume 
w.l.g. that $a_{i_{k} k}\neq 0$ for $k=3,4,\ldots,n$ (if $n>2$).  Now consider two index 
\[ 
(i,j,i_{3},i_{4},\ldots, i_{m}), \quad   (j, i, i_{3},i_{4},\ldots, i_{m}) \in S(m,n) 
\]
By the symmetry of $\A$, we have 
\[ A_{iji_{3}i_{4}\ldots i_{m}} = A_{j i i_{3}i_{4}\ldots i_{m}} \]
which, by (\ref{eq: prfsymrk101}), is equivalent to 
\[ a_{i1}a_{j2}a_{i_{3}3}a_{i_{4}4}\ldots a_{i_{m}m} = a_{j1}a_{i2}a_{i_{3}3}a_{i_{4}4}\ldots a_{i_{m}m} \]
It follows that  
\beq\label{eq: prfsymrk102}
a_{i1}a_{j2} = a_{j1}a_{i2} ,\quad  \forall  i, j \in [n]
\eeq
Thus we have 
\beq\label{eq: prfsymrk103}
\frac{a_{i1}}{a_{i2}} = \frac{a_{j1}}{a_{j2}} ,\quad    \forall  (i, j): a_{i2}a_{j2}\neq 0. 
\eeq
It follows by (\ref{eq: prfsymrk103}) that $\rank(A)=1$.  So we may write $\al_{k}=\mu_{k}\al_{1}$ for all $k\in \dbra{m}$ 
($0\neq \mu_{j}\in \R$) with $\mu_{1}=1$.  Denote $\by:=\al_{1}\in \R^{n}$ which is a nonzero vector, and  
$\la_{j}: =\mu_{1}\mu_{2}\ldots \mu_{m}$. Then (\ref{eq: symrk1}) is proved. 
\end{proof}

\indent  It is easy to see from Lemma \ref{le: 301} that a linear symmetric rank-1 preserver $\phi$ on $\aut(\T_{m,n})$ can be uniquely 
determined by a linear positive mapping, say $\pi$, on $\R^{n}$, i.e., $\pi(\bx)=\hat{\bx} = M \bx$ for some invertible matrix 
$M=(m_{ij})\in \R^{n\times n}$.  Now we express our main result on the multlinear symmetric rank preserver. 
\begin{thm}\label{th: 301}
Let $\phi$ be a multi-linear map on $\ST_{m,n}$. Then $\phi\in \aut(\ST_{m,n})$ is a linear symmetric rank preserver if and only if there exists an invertible 
matrix $B\in \RPlus^{n\times n}$ such that 
\beq\label{eq: phi_symrkposi}
\phi (\X) = \dbra{B}\cdot{} \X
\eeq 
for any symmetric tensor $\X\in \T_{m,n}$. 
\end{thm} 

\begin{proof}
For the necessity, we let $\phi\in \aut(\ST_{m,n})$ be a symmetric rank preserver. By Lemma \ref{le: 301}  we know that $\phi$ is determined by its 
projection on $\R^{n}$, say $\pi\in \aut(\R^{n})$. Now we write $B=[\be_{1},\be_{2}, \ldots, \be_{n}]\in \R^{n\times n}$ with $\be_{j}=\pi(\bfe_{j})$ for $j\in [n]$. 
Here $\bfe_{j}\in \R^{n}$ is the $j$th coordinate vector of $\R^{n}$. It is obvious that $B$ is invertible since $pi$ is an automorphism in $\R^{n}$.  Now we write 
$\X$ as 
\[ \X = \sum\limits_{i_{1},i_{2},\ldots, i_{m}} X_{i_{1}i_{2}\ldots i_{m}} \bfe_{i_{1}}\times\bfe_{i_{2}}\times\ldots \times \bfe_{i_{m}} \]
where $\bfe_{j}\in \R^{n}$ is the $j$th coordinate vector for each $j\in [n]$.  By the linearity we get 
\beyy 
\phi(\X)  &=& \sum\limits_{i_{1},i_{2},\ldots, i_{m}} X_{i_{1}i_{2}\ldots i_{m}} \phi(\bfe_{i_{1}}\times\bfe_{i_{2}}\times\ldots \times \bfe_{i_{m}}) \\
             &=&  \sum\limits_{i_{1},i_{2},\ldots, i_{m}} X_{i_{1}i_{2}\ldots i_{m}} \pi(\bfe_{i_{1}})\times\pi(\bfe_{i_{2}})\times\ldots \times \pi(\bfe_{i_{m}})\\
             &=& \sum\limits_{i_{1},i_{2},\ldots, i_{m}} X_{i_{1}i_{2}\ldots i_{m}}  \be_{i_{1}}\times\be_{i_{2}}\times\ldots \times \be_{i_{m}}
\eeyy 
which is followed by $\phi(\X) = \dbra{B}\cdot{} \X$. \\
\indent Now we prove the sufficiency. Suppose that $\phi\in \aut(\ST_{m,n})$ is a linear map satisfying (\ref{eq: phi_symrkposi}). Denote 
$\hat{\X}:=\phi(\X)=(\hat{X}_{i_{1} i_{2} \ldots i_{m}})$.  We let $\pi$ be the projection of $\phi$ on $\R^{n}$ as defined above, and let (\ref{eq: symrk1decomp} ) 
be the CP decomposition of $\X$. Then for any $(i_{1},i_{2},\ldots, i_{m})\in S(m,n)$, we have by (\ref{eq: phi_symrkposi}) 
\beyy
\hat{X}_{i_{1} i_{2} \ldots i_{m}}  
             &=& \sum\limits_{j_{1},j_{2},\ldots, j_{m}} X_{j_{1} j_{2}\ldots j_{m}} b_{i_{1}j_{1}}b_{i_{2}j_{2}}\ldots b_{i_{m}j_{m}}\\
             &=& \sum\limits_{j=1}^{R} [\sum\limits_{j_{1},j_{2},\ldots, j_{m}} x_{j_{1}j} x_{j_{2} j}\ldots x_{j_{m} j} b_{i_{1}j_{1}}b_{i_{2}j_{2}}\ldots b_{i_{m}j_{m}}] \\
             &=& \sum\limits_{j=1}^{R} (\hal_{j}^{m})_{i_{1} i_{2} \ldots i_{m}}  
\eeyy
It follows that  
\[ \hat{\X} = \sum\limits_{j=1}^{R}  \hal_{j}^{m}  \]
which means that $\phi\in \aut(\ST_{m,n})$ is a linear symmetric rank preserver induced by $\pi$.  
\end{proof}     

\indent  Let $B^{(k)} =(b^{(k)}_{ij}) \in \R^{p_{k}\times q_{k}}$ for $k\in \dbra{m}$. We define the $2m$-order tensor $\B:=\B(B^{(1)},\ldots, B^{(m)})$ as 
\[ \B=B^{(1)}\times B^{(2)} \times\ldots \times B^{(m)} =[B_{i_{1}\ldots i_{m}; j_{1}\ldots j_{m}}]  \]
whose entries are defined by 
\beq\label{eq: defBtensor}  
B_{i_{1}\ldots i_{m}; j_{1}\ldots j_{m}} = b^{(1)}_{i_{1}j_{1}}b^{(2)}_{i_{2}j_{2}}\ldots b^{(m)}_{i_{m}j_{m}} 
\eeq
$\B:=\B(B^{(1)},\ldots, B^{(m)})$ is called the \emph{tensor product} of $(B^{(1)},\ldots, B^{(m)})$, and is of size 
$p_{1}\times \ldots \times p_{m}\times q_{1}\times \ldots \times q_{m}$. Then we have 
\beq\label{eq: mB2tensor}
\dbra{B_{1},B_{2},\ldots, B_{m}}\cdot{} \A = \B\times \A
\eeq
\indent  For $B_{1}=\ldots =B_{m}=B\in \R^{n\times n}$, we write $\B:=\B(B,\ldots, B)\in \T_{2m;n}$. By Theorem \ref{th: 301}, we get 
\beq\label{eq: unifysympreserv}  
\phi(\X) = \B\times \X,  \quad  \forall \X\in \ST_{m; n}   
\eeq
For $m=2$, (\ref{eq: unifysympreserv}) turns out to be $\phi(X)=BXB^{\top}$. It is immediate from Theorem \ref{th: 301} that each column of $B$ is exactly 
the image of the projection of $\phi$ on $\R^{n}$.  From Theorem \ref{th: 301} we obtain 

\begin{cor}\label{co: rkpres01}
Let $\phi\in \aut(\C^{n\times n})$. Then $\phi$ is a symmetric rank preserver  if and only if there exist invertible matrix $P\in \C^{n\times n}$ such that
\beq\label{eq: rkpres0101}
\phi(X)=X\times_{1} P^{\top}\times_{2} P, \quad \forall X\in \C^{n\times n} 
\eeq
\end{cor} 
\indent  Formula (\ref{eq: rkpres0101}) in Corollary \ref{co: rkpres01} in the matrix form is $\phi(X)=P^{\top}XP$ which is exactly the form for a (linear) 
symmetric rank preserver.  Now if $\phi$ is also required to be a positive preserver (i.e., preserving the entrywise nonnegativity of a tensor), the we 
have     

\begin{cor}\label{co: 302}
Let $\phi\in \aut(\ST_{m,n})$ be a nonnegative linear symmetric preserver.  Then $\phi$ fixes the identity tensor if and only if  $B$ is a 
permutation matrix  in (\ref{eq: phi_symrkposi}), if and only if  the projection of $\phi$ on $\R^{n}$ preserves the set $\set{\bfe_{1},\bfe_{2},\ldots, \bfe_{n}}$. 
\end{cor} 

\begin{proof}
Let $\phi\in \aut(\ST_{m,n})$ be a nonnegative linear symmetric preserver that fixes the identity tensor. Then by Theorem \ref{th: 301} there exists an invertible matrix $B=[\be_{1},\ldots, \be_{n}]\in \R^{n\times n}$ such that (\ref{eq: phi_symrkposi}) holds. Thus we have 
\beq\label{eq: phi_rk1ident03}
\caI = \caI \times_{1}B\times_{1}B\times_{2} \ldots \times_{m}B 
\eeq 
since $\phi$ preserves the identity tensor. It follows 
\beq\label{eq: corprf30201}
 \caI = \sum\limits_{j=1}^{n} \be_{j}^{m}
\eeq
By the linearity of $\phi$, we have 
\beq\label{eq: thproof30102}
 \caI =\phi(\caI) =\sum\limits_{j=1}^{n} \phi(\eps_{j}^{m}) = \sum\limits_{j=1}^{n} \be_{j}^{m} 
\eeq
Denote $B=[b_{ij}]\in \R^{n\times n}$. By the nonnegativity of $\be_{j}$ and (\ref{eq: thproof30102}), we have 
\beq\label{eq: thproof30103}
 (\be_{j}^{m})_{i_{1}i_{2}\ldots i_{m}} =b_{i_{1}j}b_{i_{2}j}\ldots b_{i_{m}j} =0
\eeq
for all $j\in [n]$ and all $(i_{1},i_{2},\ldots, i_{m})\in S(m,n)$ where $i_{1},i_{2},\ldots, i_{m}$ are not identical. 
It follows that each $\be_{j}$ is a coordinate vector. It turns out that $B\in \R^{n\times n}$ shall be a permutation matrix since by  
(\ref{eq: thproof30102}) each row of $B$ has a unique one. \\ 
\indent  Conversely, we suppose that (\ref{eq: phi_symrkposi}) holds with $B\in \R^{n\times n}$ being a permutation matrix.  We may assume that $B$
corresponds to a permutation $\tau\in S_{n}$, i.e., $b_{ij}=1$ if and only if $j=\tau(i)$ for each $i$.  Denote 
$\hat{\caI}:=\phi(\caI)=(\hat{I}_{i_{1}i_{2}\ldots i_{m}})$. For each $(i_{1}, i_{2}, \ldots , i_{m})\in S(m,n)$, we have 
\beyy
\hat{I}_{i_{1} i_{2} \ldots i_{m}}  
             &=& \sum\limits_{j_{1},j_{2},\ldots, j_{m}} I_{j_{1} j_{2}\ldots j_{m}} b_{i_{1}j_{1}}b_{i_{2}j_{2}}\ldots b_{i_{m}j_{m}}\\
             &=& I_{\tau(i_{1}) \tau(i_{2})\ldots \tau(i_{m})}    
\eeyy
which equals 1 if and only if $\tau(i_{1})=\tau(i_{2}) = \ldots = \tau(i_{m})$ by the definition of the identity tensor. It follows that  
\[ \hat{I}_{i_{1} i_{2} \ldots i_{m}} =1 \ffi  i_{1}= i_{2}= \ldots = i_{m} \in [n] \]
Consequently $\hat{\caI}:= \caI$.  So $\phi$ fixes the identity tensor. \\ 
\indent  The second part of the corollary (i.e., the projection of $\phi$ on $\R^{n}$ preserves set $\set{\bfe_{1},\bfe_{2},\ldots, \bfe_{n}}$)  can be deduced 
directly from the proof of Theorem \ref{th: 301}. 
\end{proof}     
 
\indent  Now we extend the result in Theorem \ref{th: 301} and consider a general linear rank preserver in $\aut(\T_{m,n})$.   
\begin{thm}\label{th: 408}
Let $\phi\in \aut(\T_{m,n})$ be a linear map and  $\A\in \T_{m,n}$  be a rank-$R$ tensor possessing a CP decomposition (\ref{eq: rk1decomp}).  
Then $\phi$ is a rank preserver if and only if there exist invertible matrices $B_{j}\in \R^{n}, j\in \dbra{m}$ and a permutation $\tau\in S_{m}$ such that 
\beq\label{eq: phirkpreserv}
\phi (\A) = \B\times\A\times \K^{\tau}
\eeq  
where $\B:= B_{1}\times B_{2}\times \ldots\times B_{m} \in \T_{2m,n}$ is an $2m$-order tensor defined by (\ref{eq: defBtensor}).   
\end{thm}

\begin{proof}
For the sufficiency, we let $\phi$ be defined by ({eq: phirkpreserv}). We want to prove that it is a rank preserver. It suffice to prove that it preserves the rank-1 tensor. 
But this is an obvious fact when we look at the rank-1 tensor $\A=\al_{1}\times\al_{2}\times  \ldots \times \al_{m}$ since by (\ref{eq: phirkpreserv}) we have 
\beyy\label{eq: rk1preserv}
\phi (\A)  &=& \dbra{B_{1},B_{2},\ldots, B_{m}}\cdot{} \A\times \K^{\tau} \\
              &=& \dbra{B_{1},B_{2},\ldots, B_{m}}\cdot{} \A^{\tau} \\
              &=& \hal_{i_{1}}\times\hal_{i_{2}}\times\ldots \times \hal_{i_{m}} 
\eeyy  
where $i_{k}=\tau(k)$ and $\hal_{i_{k}}=B_{k}\al_{i_{k}}\in \R^{n}$ is nonzero due to the nonsingularity of $B_{k}$. \\     
\indent  Conversely,  let $\A\in \T_{m,n}$  be a rank-$R$ tensor with CP decomposition (\ref{eq: rk1decomp}) and $\phi$ be a rank preserver. Denote by 
$\pi_{k}$ the projection of $\phi$ on direction $k$ ($k\in \dbra{m}$).  Then by Corollary 3.10 of \cite{Chooi2017} we have    
\beyy\label{eq: rkpreserv01}
\phi(\A)  &=& \sum\limits_{j=1}^{R} \phi(\al_{1j}\times \al_{2j}\times \ldots \times \al_{mj})\\
              &=& \sum\limits_{j=1}^{R} \hal_{i_{1},j}\times\hal_{i_{2},j}\times \ldots \times \hal_{i_{m},j}\\
              &=& \sum\limits_{j=1}^{R} \al_{i_{1},j}\times\hal_{i_{2},j}\times \ldots \times \hal_{i_{m},j}\\  
\eeyy
with $i_{k} =\tau(k)$ where $\tau$ is a permutation on $\dbra{m}$ and $\hal_{j}=\pi_{j}(\al_{j})$. Now we denote 
\[ \hat{\A}_{j}^{\tau} := \al_{i_{1},j}\times\hal_{i_{2},j}\times \ldots \times \hal_{i_{m},j}, \quad  \forall j\in \dbra{m} \]
By Lemma \ref{le: 403} we have $\hat{\A}_{j}^{\tau} = \hat{\A}_{j}\times \K^{\tau}$ for each $j\in \dbra{m}$. It follows that  
\beyy\label{eq: rkpreserv02}
\hat{\A} =\phi(\A)  
              &=& \sum\limits_{j=1}^{R} \hat{\A}_{j}^{\tau}\\
              &=& \sum\limits_{j=1}^{R} \hat{\A}_{j}\times \K^{\tau}\\
              &=& (\sum\limits_{j=1}^{R} \hal_{1j}\times\hal_{2j}\times \ldots \times \hal_{mj})\times \K^{\tau} \\
              &=& \dbra{B_{1},\ldots, B_{m}}\cdot{}\A \times \K^{\tau}\\
              &=& \B\times\A\times\K^{\tau}   
\eeyy
The second last equality is due to the fact that for each $j\in \dbra{m}$ there exists an invertible matrix $B_{j}\in \R^{n\times n}$ such that $\phi_{j}(\bx) =B_{k}\bx$ for all 
$\bx\in \R^{n}$ since $\phi_{j}\in \aut(\R^{n})$ (the set of linear automorphisms $\aut(\R^{n})$ of $\R^{n}$ is isomorphic to the general linear group $GL_{n}$ 
consisting of all invertible matrices in $\R^{n\times n}$), and the last equality follows from (\ref{eq: mB2tensor}). Thus there exist invertible matrices 
$B_{1}, B_{2},\ldots, B_{m}$ such that $\hal_{kj}= B_{k}\al_{k}$. The proof is completed.  
\end{proof}

\indent  We call tensor $\B=B_{1}\times \ldots \times B_{m}$  an \emph{associated tensor} with $\phi$ if $\phi_{j}(\bx) =B_{k}\bx$ for all $\bx\in \R^{n}$ where $\phi_{j}$ is the 
projection of $\phi$ on direction $j$. Note that if we fix a $\tau$, then $\phi$ is associated with an $\B$ uniquely. Conversely, from the argument above, we see that 
there are $m!$ linear rank preservers in $\aut(\T_{m;n})$ associated with an $2m$-order tensor generated by $(B_{1},\ldots, B_{m})$.  Also if we choose $\tau$ to be the 
identity map on $\dbra{m}$, then by (\ref{eq: phirkpreserv}) the associated linear rank preserver $\phi$ acts on $\T_{m;n}$ in form 
\beq
\phi(\X) = \B\times \X,  \quad  \forall \X\in \T_{m;n}
\eeq  
\indent  In order to describe the linear identity and rank preserver,  we let $B^{(k)}=(b^{(k)}_{ij})\in \R^{m\times n}, \forall k\in \dbra{p}$, and let 
$\B:=(B_{ijk})\in \R^{m\times n\times p}$ be the 3-order tensor defined by $B(:,:,k)=B^{(k)}$ for each $k\in \dbra{p}$, and define 
$\hat{\B}:= (\hat{B}_{i_{1}i_{2}\ldots i_{p}})$ as the Hadamard product of $(B^{(1)}, B^{(2)}, \ldots, B^{(p)})$, i.e., 
\[ \hat{B}_{i_{1}i_{2}\ldots i_{p}}= \sum\limits_{j=1}^{n} b^{(1)}_{i_{1}j}b^{(2)}_{i_{2}j}\ldots b^{(1)}_{i_{1}j} \]
Here each entry of $\hat{\B}$ can be regarded as the Hadamard product of the corresponding $p$ mode-3(along the third direction) slices of tensor $\B$.  
When each $B^{(k)}\in \R^{n\times n}$ is a permutation matrix, say, corresponding to a permutation $\pi_{k}$, then we have   
It follows immediately from that 
 \begin{cor}\label{co: 409}
Let $\phi$ be a linear positive rank preserver on $\T_{m,n}$ associated with the identity map $\caI$.  Then it fixes the identity tensor if  
tensor $\B$ associated with $\phi$ is an identity tensor, i.e.,  
\beq\label{eq: rkidentpreserv}
\B\times \caI =\caI
\eeq  
\end{cor} 

\begin{cor}\label{co: 403}
Let $\phi\in \aut(\R^{n\times n})$ be a linear map.  Then $\phi$ is a rank preserver if and only if  there is invertible matrices $B_{1}, B_{2}\in \R^{n\times n}$
such that 
\beq\label{eq: grkpres}
\phi (X) = X\times_{1} B_{1}\times_{2} B_{2} 
\eeq 
for any real matrix $X\in \R^{n\times n}$.
\end{cor} 
Note that formula (\ref{eq: grkpres}) is exactly the unified form of two formulae (\ref{eq: pres01}) and (\ref{eq: pres02}).

\section*{Acknowledgement}
This work was partially supported by Hong Kong Research fund(No. PolyU 502111, 501212).


\vskip 50pt
\begin{thebibliography}{99}

\bibitem {CS2013}
Dustin Cartwright, Bernd Sturmfels, \textit{The number of eigenvalues of a tensor}, 
Linear Algebra and Its Applications, 438(2013) 942-952.
   
\bibitem{cqcauchy2014}
H. Chen and L. Qi, 
\textit{Positive definiteness and semi-definiteness of even order symmetric Cauchy tensors}, to appear in:
Journal of Industrial and Management Optimization, arXiv:1405.6363. 

\bibitem{Chooi2017}
W. L. Chooi, K. H. Kwa, M.-H. Lim, 
\textit{Coherence invariant maps on tensor products}, 
Linear Algebra and Its Applications, 516(2017) 24-46. 

\bibitem{cglm2008}
P. Comon, G. Golub, L.-H. Lim, and B. Mourrain, 
\textit{Symmetric tensors and symmetric tensor rank}, 
SIAM. J. Matrix Analysis and Applications, 30(2008) 1254-1279. 

\bibitem{frob1897}
\emph{G. Frobenius, \"{U}ber die Darstellung der endlichen Gruppen durch lineare Substitutionen},
S.-B. Press,  Akad. Wiss. Berlin, 994-1015 (1897).

\bibitem{gron1977}
R. Grone, 
\emph{Decomposable tensors as a quadratic variety}, 
Proc. Amer. Math. Soc. 64(2) (1977) 227-230.

\bibitem{hua1951}
L.-K. Hua, 
\emph{A theorem on matrices over a s-field and its applications, Acta Math},
Sinica 1(1951) 109-163.

\bibitem{hhlqdet2011}
S. Hu, Z. Huang, C. Ling and L. Qi, 
\textit{On determinants and eigenvalue theory of tensors}, Journal of Symbolic Computation,  50(2013) 508-531. 

\bibitem{HQ2018}
Z. Huang, and L. Qi, 
\textit{On determinants and eigenvalue theory of tensors}, Journal of Symbolic Computation,  50(2013) 508-531. 

\bibitem{HSS2017}
Zejun Huang, Shiyu Shi, and Nung-Sing Sze,
\textit{Linear rank preservers of tensor products of rank one matrices},
Linear Algebra and Its Applications, 508(2016) 255-271.

\bibitem{jacob1955}
H.G. Jacob, 
\textit{Coherence invariant mappings on Kronecker products}, 
Amer. J. Math. 77(1) (1955) 177-189.

\bibitem{kolda2009}
T. G. Kolda, B. W. Bader, 
\textit{Tensor Decompositions and Applications}, 
SIAM Review, 2009. 

\bibitem {limeig2005}
L. -H. Lim,
\textit{Singular values and eigenvalues of tensors: A variational approach}, 
Proc. of 1st IEEE Int'l Workshop on Computational Advances in Multi-Sensor Adaptive Processing (CAMSAP),  Dec. 2005, pp. 129-132.

\bibitem {lim1976}
M.-H. Lim, 
\emph{Rank k vectors in symmetric classes of tensors}, 
Canad. Math. Bull. 19(1) (1976) 67-76.
 
\bibitem {MN1979}
J.R. Magnus and H. Neudecker, 
\textit{The Commutation matrix:  some properties and applications},
The Annuals of Statistics,  7(1979):  381-394.  

\bibitem {MM1971}
M. Marcus, 
Linear transformations on matrices, 
J. Res. Nat. Bur. Standards Sec.B, 75B, 107-113 (1971).

\bibitem {MM1971}
M. Marcus and B. N. Moyls,
Transformations on tensor product spaces, Pacific J. Math. 9(1959), 1215-1221.

\bibitem {Minc1977}
H. Minc, 
Linear transformations on matrices: rank 1 preservers and determinant preservers,
Linear and Multilinear Algebra 4: 265-272 (1977).
    
\bibitem {qieig2005}
Liqun Qi,
\textit{Eigenvalues of a real supersymmetric tensor}, 
Journal of Symbolic Computation,  40 (2005) 1302-1324.

\bibitem {qicop2013}
Liqun Qi,
\textit{Symmetric nonnegative tensors and copositive tensors}, 
Linear Algebra and Its Applications, 439(2013) 228-238.
 
\bibitem{Minc1988} 
H. Minc, 
\textit{Nonnegative Matrices},  New York: Wiley, 1988.  
 
\end{thebibliography}
\end{document}